\documentclass[11pt]{amsart}
\usepackage{amssymb,hyperref}
\usepackage{todonotes}
\usepackage{graphicx}
\usepackage[matrix,arrow]{xy}
\newcommand{\PP}{\mathbb P}

\newcommand{\CC}{{\mathbb C}}
\newcommand{\NN}{{\mathbb N}}
\newcommand{\ZZ}{{\mathbb Z}}
\newcommand{\QQ}{{\mathbb Q}}
\newcommand{\RR}{{\mathbb R}}
\newcommand{\FF}{{\mathbb F}}

\newcommand{\p}{\mathfrak p}

\newcommand{\Pic}{\mathrm{Pic}}
\newcommand{\NS}{\mathrm{NS}}
\newcommand{\MWL}{\mathrm{MWL}}

\newcommand{\tD}{2h}
\newcommand{\jD}{h}
\newcommand{\polD}{H}
\newcommand{\fH}{f_H}
\newcommand{\XH}{X_H}

\DeclareMathOperator{\rank}{rank}

\numberwithin{equation}{section}
\begin{document}

\title{Low degree rational curves on quasi-polarized K3 surfaces}

\author{S\l awomir Rams}
\address{Institute of Mathematics, Jagiellonian University, 
ul. {\L}ojasiewicza 6,  30-348 Krak\'ow, Poland}
\email{slawomir.rams@uj.edu.pl}

\author{Matthias Sch\"utt}
\address{Institut f\"ur Algebraische Geometrie, Leibniz Universit\"at
  Hannover, Welfengarten 1, 30167 Hannover, Germany}

    \address{Riemann Center for Geometry and Physics, Leibniz Universit\"at
  Hannover, Appelstrasse 2, 30167 Hannover, Germany}

\email{schuett@math.uni-hannover.de}

\dedicatory{Dedicated to Alex Degtyarev of the occasion on his 60th birthday}

\date{}
\subjclass[2010]
{Primary: {14J28};  Secondary {14J27, 14C20}}
\keywords{K3 surface,  rational curve, polarization, elliptic fibration, hyperbolic lattice,
parabolic lattice}

\begin{abstract}
We prove that there are at most	 $(24-r_0)$ low-degree rational curves on high-degree models of K3 surfaces with at most Du Val singularities, where $r_0$ is the number of exceptional divisors on the minimal resolution. 
We also provide several existence results in the above setting (i.e. for rational curves on quasi-polarized K3 surfaces), 
which imply that for various values of $r_0$ our bound cannot be improved.
\end{abstract}

\maketitle

\newcommand{\XXd}{X_{d}}
\newcommand{\XXf}{X_{4}}
\newcommand{\XXp}{X_{5}}
\newcommand{\mF}{\mathcal F}
\newcommand{\MW}{\mathop{\mathrm{MW}}}
\newcommand{\mL}{\mathcal L}
\newcommand{\mR}{\mathcal R}
\newcommand{\Ruledeight}{S_{11}}
\newcommand{\Ruledfour}{S_{4}}
\newcommand{\DivisorRest}{\mathfrak Rest}
\newcommand{\Pl}{\Pi}
\newcommand{\reg}{\operatorname{reg}}

\newcommand{\IK}{{{\mathrm I}}}
\newcommand{\II}{{\mathop{\rm II}}}
\newcommand{\III}{{\mathop{\rm III}}}
\newcommand{\IV}{{\mathop{\rm IV}}}

\theoremstyle{remark}
\newtheorem{obs}{Observation}[section]
\newtheorem{rem}[obs]{Remark}
\newtheorem{example}[obs]{Example}
\newtheorem{ex}[obs]{Example}
\newtheorem{conv}[obs]{Convention}
\newtheorem{crit}[obs]{Criterion}
\newtheorem{claim}[obs]{Claim}
\newtheorem{Fact}[obs]{Fact}
\newtheorem{Addendum}[obs]{Addendum}

\theoremstyle{definition}
\newtheorem{Definition}[obs]{Definition}
\theoremstyle{plain}
\newtheorem{prop}[obs]{Proposition}
\newtheorem{theo}[obs]{Theorem}
\newtheorem{Theorem}[obs]{Theorem}
\newtheorem{lemm}[obs]{Lemma}
\newtheorem{cor}[obs]{Corollary}
\newtheorem{assumption}[obs]{Assumption}

\newcommand{\ux}{\underline{x}}
\newcommand{\ud}{\underline{d}}
\newcommand{\ue}{\underline{e}}
\newcommand{\mmS}{{\mathcal S}}
\newcommand{\mmP}{{\mathcal P}}
\newcommand{\nlines}{\mbox{\texttt l}(\XXp)}
\newcommand{\ii}{\operatorname{i}}

\newcommand{\nonlinflec}{{\mathcal D}}
\newcommand{\linflec}{{\mathcal L}}
\newcommand{\flec}{{\mathcal F}}

\def\Fn{\operatorname{Fn}}
\def\extended{^\text{\rm ex}}
\def\Fex{\Fn\extended}

\newcommand{\Gfamma}{\Fex_{d}(X,H)}
\newcommand{\SG}{G}


\section{Introduction}
\label{intro}
It is well-known that a complex projective normal cubic surface 
with non-empty singular locus contains at most 21 lines 
(or infinitely many -- see e.g. \cite[Table~1]{sakamaki}), 
i.e.\ the number of lines on such a surface (if finite) 
is strictly less than the  number of lines on a smooth cubic surface.
Similar phenomena can be observed for lines on low-degree complex K3 surfaces 
(the number of lines on a complex K3 quartic, resp. sextic, resp. octic with singularities is always strictly lower than the maximal number for smooth K3 surfaces of the same degree, see \cite{DR-2021}, \cite{DR-2021-6}, \cite{DR-2023}), whereas it does not happen on K3 quartics in characteristic $p=2$
(cf.\ \cite{RS-2}).

In this note  we allow the surfaces in question to have singular points, i.e.\ we
consider  birationally
quasi-polarized  K3 surfaces of a fixed degree (i.e.\ pairs $(X,H)$, such that $H \in \mbox{Pic}(X)$ is big, nef, base-point free and non-hyperelliptic with  $H^2=\tD$, cf.\ \cite[\S ~2.3]{DR-2021} and Criterion \ref{crit}). 
For a non-negative integer $d$ we define
\[
r_d := r_d(X) := \#\{\text{rational curves } C\subset X \text{ with } C.H = d\},
\]
The problem of finding the maximum
of $r_d$ for very ample $H$ of a fixed degree 
 has a long history   (cf.\ \cite{degt}, \cite{DIS}, \cite{RS}, \cite{Segre}).
 In particular,  
for surfaces of small degree (smooth or with singular points), the behaviour of  $r_d$ seems to follow no pattern.
For polarized complex K3 surfaces of  degree  $H^2>4d^2$, where $H$ is very ample,
Miyaoka  \cite{Miyaoka} applied the orbibundle Miyaoka--Yau--Sakai inequality from \cite{Miyaoka-orbi}
to obtain the following  bound: 
\begin{eqnarray}
\label{eq:Miyaoka}
\frac{1}{d} r_1+ \frac{2}{d} r_2+\hdots+ r_d \leq\frac{24 \jD}{	\jD-2d^2}.
\end{eqnarray}
  For $\jD>50d^2$ and all positive $i\leq d$, this implies the inequality
\[
r_i\leq 24. 
\]
Subsequently, for very ample $H$,
a lattice-theoretic approach led to cha\-ra\-cte\-ris\-tic-free bounds on the
integers
\[
S_d := r_1+\hdots+r_d = \#\{\text{rational curves } C\subset X \text{ with } 0 < \deg(C)\leq d\},
\]
and a characterization of the K3 surfaces attaining them for large $H^2$ (see \cite[Thm.~1.1-1.3]{RS-24}).

We now turn to quasi-polarized K3 surfaces $X$ where $H$ is still big and nef such that $|H|$ induces a birational morphism,
but it may contract $(-2)$-curves to Du Val singularities, or $ADE$ singularities 
(namely those smooth rational curves $C$ with $C.H=0$, i.e.\ curves 'of degree zero').
In this setting, the number $r_0$ of  curves contracted equals the rank of the root lattice generated by the components
of the minimal resolution of the ADE-singularities (this resolution is exactly $X$).
There is extensive literature on exceptional curves of such a minimal resolution. One has $r_0 \leq 19$ for $p=0$ and the sharp bound $r_0 \leq 21$ when $0< p \leq 19$
(resp.  $r_0 \leq 20$ when $p > 19$)  -- see \cite[Thm~1.1]{Shimada}. 
Still, hardly anything is known 
 on the impact  of the number $r_0$  of degree-$0$ rational curves on the maximal value of the integer $S_d$ for quasi-polarized K3 surfaces.  
 
 Here we adapt the approach from 
 \cite{RS-24}
  to  obtain  bounds for the numbers $S_d$
  when $|H|$ contracts some curves ($r_0>0$).
More precisely, we show 
the following theorem.

\begin{Theorem}
\label{thm}
Let $d\in\NN$ and let $p \neq 2,3$.
\begin{enumerate}
\item[(i)]
For all $\jD \gg 0$ and for all 
quasi-polarized K3 surfaces $X$ of degree $\tD$ over a
 field $k$ of characteristic $p$,  
one has
\[
S_d\leq (24 - r_0) .
\]
\item[(ii)]
If $\jD\gg 0$ and $S_d>(21-r_0)$, 
then the  rational curves of degree at most $d$ are fibre components
of a genus one fibration.
\item[(iii)]
For $d\geq 3$, let $\jD\geq d^2-1$. Then
for any   $r\in
\{1,\hdots, 17\}$,
there are K3 surfaces of degree $2h$ over $\CC$ such that 
$$
r_d\geq 24-r \;\;\;  \mbox{ and } \;\;\;  r_0 =r.$$
\item[(iv)]
Let $p\equiv 1\mod 4$ prime.
For $d\geq 3$, 
the statement of (iii) holds true
for all $\jD\geq  d((p_0+1)d+1-p_0)/2$ 
where $p_0 = \min\{q \text{ prime}; \; q\nmid 2d\}$.
 \item[(v)]
 For $d\geq 3$ and primes $p\equiv 3\mod 4$, 
 the statement of (iii) holds true under either of the following conditions:
 \begin{itemize}
 \item
 $p\nmid d$ and $h\geq d((p+1)d+1-p)/2$;
 \item
$p>4\sqrt{d+1}$ and  $h\geq d^2-1$.
 \end{itemize}
\end{enumerate}
\end{Theorem}

\noindent
It is natural to ask whether the claims (i) and (ii)
of the above theorem hold over any field of characteristic $p$  as soon as $h > 42 d^2$ and $p \neq 2,3$ 
(as in \cite{RS-24}) -- we discuss this briefly in Remark~\ref{rem:effective}.

The outline of the paper is as follows. After presenting the set-up in Section~\ref{s:setup}, we recall the notion of extended $d$-Fano graph of a quasi-polarized K3 surface and discuss properties of its subgraphs in Section~~\ref{s:prep}. 
In the subsequent section, we recall the notion of a geometric (in a fixed degree) graph, and use the so-called intrinsic polarization 
(cf.\ \cite{degt}) to show that a hyperbolic graph can be geometric in at most finitely many degrees (Corollary~\ref{cor:finite}). After those preparations we give a proof of parts (i), (ii) of Theorem~\ref{thm}. 

 In Section~\ref{s:pf-0} we prove the existence of K3 surfaces whose Neron-Severi contains the lattice 
$U \oplus A_r \oplus \langle -2c_0\rangle$, where $c_0\in\ZZ_{\geq 2}$ using the theory of elliptic fibrations (Thm~\ref{thm:c_0}). 
In the complex case, this leads to a proof of Thm~\ref{thm}~(iii) in \S \ref{ss:pf}. 
The case of positive characteristic $p >3$ requires extra care to avoid that the fibrations degenerate;
this is discussed in Section~\ref{s:pf-p}, where we prove claims (iv), (v) of  Thm~\ref{thm}. 

In specific situations, the bounds in Theorem \ref{thm} (iv) and (v) can be improved drastically.
We discuss this for some cases in Section \ref{s:specific} (see Proposition \ref{prop:small_d});
in Theorem \ref{thm:r<=14} we also work out unconditional results with better bounds for $h$ when  $r\leq 14$.

\begin{conv}
	We assume that the base field $k$ of characteristic $p\neq 2,3$
	is algebraically closed. Indeed, Theorem~\ref{thm} stays valid under base extension, so we can make the above assumption  without loss of generality. 
	
	All  curves are assumed to be irreducible.
\end{conv}


\section{Set-up} \label{s:setup}

Let $(X,H)$ be a  
quasi-polarized K3 surface of degree $\tD$, 
i.e.\ a pair  where $X$ is a smooth K3 surface  
over an algebraically closed field $k$ of characteristic $p$, and  $H$ 
is a big, nef, base-point free  divisor of square $H^2=\tD$.

Recall that for  a big, nef, and base-point free divisor $H$,  
the linear system $|H|$ defines a
morphism
\[
\fH \, : \, X \rightarrow \XH \subset \PP^{\jD+1}
\]
which is either birational or 
of degree $2$ (in the latter case  one speaks of a hyperelliptic linear system, see \cite[p.~615]{S-D}). If the  map $\fH$ is birational, then it is the minimal resolution of  singularities of its image $\XH$. Moreover, 
$\XH$  is contained in no hyperplane of $\PP^{\jD+1}$
and all its singularities are at most Du Val. In this case we say that $(X,H)$ is a  
{\bf birationally quasi-polarized K3 surface of degree $\tD$}. 
  By abuse of notation we will write $X$ instead of $(X,H)$ whenever it leads to no ambiguity.

The methods presented in \cite{S-D} give a direct way to
check  whether a given divisor $H$ on $X$ with $H^2=\tD>0$ is quasi-ample and non-hyperelliptic.
Indeed,  one has:

\begin{crit}
\label{crit}
Let $p \neq 2$ and let  $H$ be 
a  nef divisor on a K3 surface $X$. 
If
\begin{enumerate}
\item
$H.E>2$ for every irreducible curve $E\subset X$ of arithmetic genus $1$;
\item
$H^2\geq 4$, and for $H^2=8$, the divisor  
$H$ is not $2$-divisible in $\Pic(X)$,
\end{enumerate}
then $H$ is quasi-ample and non-hyperelliptic.
\end{crit}


\section{Extended Fano graphs}
\label{s:prep}

Given a quasi-polarized K3 surface $(X,H)$ of degree $2h > 2$ and a positive integer $d$, 
we define its {\bf extended $d$-Fano graph} (cf.\ \cite{degt}, \cite{RS-24}, \cite{DR-2021}) as the set
\[
\Gfamma = \{\text{rational curves } C\subset X \text{ with } C.H \leq d\}
\]
with multiple edges corresponding to the intersection numbers $C.C'$ for $C, C'\in\Gfamma$ and no loops.
As in \cite{DR-2021} vertices of the graph $\Gfamma$ are colored 
according to the value $d_C=C.H$
(observe that, contrary to \cite{RS-24}, in this note we allow $d_C=0$;
formally, we should also note the squares $C^2$ for the vertices of $\Gfamma$,
but later we will mostly reduce to the case where $C^2=-2$, see e.g.\ Lemma \ref{lem:hyper}).

From the Hodge Index Theorem we immediately obtain (cf.\ \cite[(6.2)]{RS-24})
\begin{eqnarray}
\label{eq:C^2}
C^2 \leq \frac{(C.H)^2}{H^2} 
\end{eqnarray}
Moreover, we claim that for $C, C' \in \Gfamma$ the following inequality holds: 
\begin{equation} \label{eq-bezout-modified}
C.C' \leq (d_C+1) (d_{C'}+1) \, .
\end{equation} 
Indeed,  for positive $d_C, d_{C'}$ 
the claim follows from  \cite[inequality (6.3)]{RS-24}, 
whereas for $d_C = d_{C'} = 0$,
both curves are exceptional, so $C \neq  C'$ intersecting implies  $C.C'=1$ (by the ADE classification of  Du Val singularities).  Finally,
for $d_{C} > d_{C'} = 0$ we
consider the Gram matrix of $H$, $C$, $C'$. By 
 the Hodge Index Theorem its determinant is non-negative
 and
 the inequality  \eqref{eq-bezout-modified}
follows quite easily taking into account that $d_{C}=1$ implies $C^2=-2$.

An analogous argument yields 
 \begin{equation} \label{eq-bezout-positive}
 C.C'\leq \frac{d_C d_{C'}}{h} \;\; \mbox{ for non-negative } C^2, C'^2.
 \end{equation}

Recall (cf.\ \cite{degt}, \cite{RS-24}) that
the intersection number induces the bilinear form on 
the formal group 
$$M(X,H):=\ZZ\Gfamma\subset\mathrm{Div}(X),
$$
that we will denote as $M$ when it leads to no ambiguity.
Moreover,  for a subgraph $\SG \subset \Gfamma$  we put $M|_{\SG}$ to denote the sublattice of $M$ generated by 
the vertices of $\SG$.

If  $M|_{\SG}\otimes\RR$ is negative-definite (resp.
 negative semi-definite, with 
 nontrivial kernel) 
 we say that $M|_{\SG}$ 
  is {\bf elliptic} (resp. {\bf parabolic}). 
Finally, when $M|_{\SG}\otimes\RR$ has a one-dimensional positive-definite subspace
and none of greater dimension
we call $M|_{\SG}$ 
 {\bf hyperbolic}.

One can easily check that the colors $d_C = C.H$ 
play no role in the arguments of \cite[\S 4-5]{RS-24}
(i.e.\ once we fix a family of rational curves on the 
K3 surface $X$ such that their graph  fails to be hyperbolic, we can repeat all the arguments from 
loc.\ cit.\
without any modifications).
For the convenience of the reader we collect the necessary facts (that follow directly from the arguments in
 loc.\ cit.\ 
 below.)

\begin{lemm}  {\bf \rm (\cite[\S 4]{RS-24})}
\label{lem:elliptic}
Let $\SG \subset \Gfamma$ be a  subgraph.
	If $M|_{\SG}$ is elliptic, then it is an orthogonal sum of finitely many Dynkin diagrams (ADE-type). Moreover, 
	we have the inequality
	$$
	\#\SG \leq \rank M|_{\SG} \leq 21 \, .
	$$
\end{lemm}
\begin{proof}
See \cite[Lemma~4.1]{RS-24} and \cite[Corollary~4.2]{RS-24} 
\end{proof}
In this note we assume  $p \neq 2,3$, so the general fibre of any genus-one fibration on $X$ is smooth (see \cite{Tate-genus}), and the number of fibre components of such a fibration is bounded by the Euler--Poincar\'e characteristic  $e(X) = 24$.

\begin{lemm}  {\bf \rm (\cite[\S 5]{RS-24})}
\label{lem:parabolic}
Let $\SG \subset \Gfamma$ be a  subgraph.
If $M|_{\SG}$ is parabolic, then it is an orthogonal sum of finitely many Dynkin diagrams
	and at least one isotropic vertex or extended Dynkin diagram ($\tilde A\tilde D\tilde E$-type),
	the latter of which are again finite in total number.
	Consequently, 
	there exists a genus one fibration
	\begin{eqnarray}
	\label{eq:fibr}
	X \to \PP^1
	\end{eqnarray}
	such that
	\[
	\SG =  \{\text{rational fibre components $\Theta$ of \eqref{eq:fibr} with } \deg(\Theta)\leq d\}.
	\]
	In particular, for such a subgraph $\SG \subset \Gfamma$, the following inequalities hold:  
	\begin{eqnarray}
	\label{eq:<=24}
\mbox{ \hspace*{2ex}}	\#\SG \leq \# \{\text{rational fibre components of \eqref{eq:fibr}}\} \leq  24.
	\end{eqnarray}
\end{lemm}
\begin{proof}
	See \cite[Lemma~5.1]{RS-24} and \cite[Lemma~5.3]{RS-24}. 
\end{proof}

Finally, let us state two lemmata that we will also need for the proof
of Theorem~\ref{thm}.

\begin{lemm}  
	\label{lem:hyper}
If $H^2 > 4 d^2$ and $M(X,H)$ is hyperbolic, 
then
$$
C^2 = -2 \quad \forall \, C \in \Gfamma.
$$
\end{lemm}	
\begin{proof} By \eqref{eq:C^2} we have
	 \begin{eqnarray}
	 \label{eq:selfsmal}
	 C^2 \in\{0, -2\} \;\;\; \forall \, C\in\Gfamma,
	 \end{eqnarray}
	 Assume that there is an isotropic vertex $C_0 \in \Gfamma$. We claim that 
	  \begin{eqnarray}
	 \label{eq:isotropicdisjoint}
	  C_0.C = 0  \;\;\; \forall \, C\in\Gfamma
	 \end{eqnarray}
Indeed let us fix $C\in\Gfamma\setminus\{C_0\}$ and  put $x := C.C_0$.	If $C^2=0$, then the determinant of the Gram matrix of $H,C_0,C$ reads:
	$$
	(-1)\cdot x \cdot (H^2\cdot x - d_{C_0} \cdot d_{C}) 
	$$ 
and it is strictly negative for $x > 0$, as the assumptions yield the inequality $(H^2\cdot x - d_{C_0} \cdot d_C) \geq 2d^2$ . This is impossible by the Hodge Index Theorem. 

Thus we have  $C^2=-2$, and  the determinant of the Gram matrix of $H,C_0,C$ reads:
$$
-H^2 \cdot x +2 \cdot x \cdot d_{C_0} \cdot d_{C} + 2 \cdot  d_{C_0}^2
$$
which is negative, within the given range of $H^2$, for $x >0$. Therefore \eqref{eq:isotropicdisjoint} follows from
 the Hodge Index Theorem. 
 
But then,  \eqref{eq:isotropicdisjoint} implies that 
$M(X,H)$ is parabolic, which is the desired contradiction.  Thus no vertex of $\Gfamma$ is isotropic.
\end{proof}	

\begin{lemm}  
	\label{lem:extendeddynkin}
	If $H^2 > 4 d^2$ and $\#\Gfamma	> 24$, then  $\Gfamma$
contains an extended Dynkin diagram.
\end{lemm}
\begin{proof}
As we have shown in the previous lemma, we have $C^2 = -2$ for all   $C\in\Gfamma$.  
Then, the Hodge Index Theorem 
can be seen to imply that
\begin{eqnarray*}
C.C' \in\{0,1,2\} \;\; \text{ for all } \;\; C\neq C'\in\Gfamma. 
\end{eqnarray*}
If we have a pair of curves in $\Gfamma$ with $C.C'=2$ then  they form the extended Dynkin diagram $\tilde A_1$.  Otherwise,
the claim follows from the well-known fact that any simple graph that is not a Dynkin
diagram contains an extended one.
 \end{proof}

\section{Proof of Theorem~\ref{thm} (i), (ii)}

 Given a colored graph $\Gamma$ without loops we define a bilinear form on the lattice
$$ 
M_{\Gamma}:= \ZZ \Gamma
$$
by defining $v.w$ as the number of vertices
between $v \neq w \in \Gamma$ and   $v^2 = -2$ for each $v \in \Gamma$. 
We call such a graph {\bf hyperbolic} (resp. 
{\bf parabolic}) iff
$M_{\Gamma}\otimes\RR$ has a one-dimensional positive-definite subspace
and none of greater dimension (resp. $M_{\Gamma}\otimes\RR$ is negative semi-definite with non-trivial kernel).

Let us fix a  hyperbolic graph $\Gamma$.
As in \cite{degt} we put $L := M_{\Gamma}/\ker(M_{\Gamma})$ and  consider the  non-degenerate lattice
\[
 L_{\polD} := (\ZZ\Gamma + \ZZ H)/\ker(\ZZ\Gamma + \ZZ H), 
\]
where $H$  satisfies the conditions
\[
H^2=\tD \mbox{ and } C.H_\Gamma=d_C \;\;\;\forall \, C\in\Gamma.
\]

One of crucial points in \cite{RS-24} is the following lemma (see also \cite[Proposition~2.8]{degt}), that remains true in our set-up.

\begin{lemm}  {\bf \rm (cf.\ \cite[Lemma~6.1]{RS-24})}
	\label{lem:embed}
	If $L_{\polD}$ is hyperbolic, then $L$ embeds into $L_{\polD}$.
\end{lemm}

\begin{proof} If we can find a vector   $w\in\ker(\ZZ\Gamma)\setminus\ker(\ZZ\Gamma + \ZZ H)$, then,  we have 
$$w.H \neq 0.$$  
Now, as in the  proof of \cite[Lemma~6.1]{RS-24} we can 
choose a vector $x\in M_{\Gamma}$ with $x^2$ positive, and
compute the determinant of the Gram matrix of $w$, $x$ and $H$. The latter is negative, as it reads 
	$$-x^2(w.H)^2,$$
 	 so  the contradiction follows. Thus 
 	 we have shown the inclusion
		 \[
	 \ker(\ZZ\Gamma)\subset\ker(\ZZ\Gamma + \ZZ H).
	 \]
	 and the proof is complete.
\end{proof}

Given a graph $\Gamma$ we can check 
whether there is a vector $H_\Gamma \in L\otimes\QQ$ that satisfies
the conditions
$$
C.H_\Gamma=d_C \;\;\;\forall \, C\in\Gamma
$$
 (observe that the above system of equations may be overdetermined). If such a vector $H_\Gamma$ exists, we call   it {\bf an intrinsic polarization} (cf.\ \cite{degt}). 
By definition, there is at most one $H_\Gamma$, so the (rational) number  $H_\Gamma^2$ is determined by the colored graph $\Gamma$, provided the intrinsic polarization exists.

Lemma~\ref{lem:embed} enables us to repeat
the argument from \cite{degt} (see also 
\cite[Proposition~6.2]{RS-24}): 
for a hyperbolic  $L_{\polD}$
we consider
the orthogonal decomposition
\[
L_{\polD}\otimes\QQ = (L\otimes\QQ) \perp (L^\perp\otimes\QQ).
\]
and represent $H$ as 
\[
H = H_\Gamma + H_\Gamma^\perp, \mbox{ where } H_\Gamma\in L\otimes\QQ, \;\; H_\Gamma^\perp\in L^\perp\otimes\QQ.
\]
Since $L^\perp$ is either zero or negative-definite,
the (rational) number  $(H_\Gamma^\perp)^2 \leq 0$ is non-positive. In this way, we obtain the following proposition.

\begin{prop}  {\bf \rm (cf.\ \cite[Proposition~6.2]{RS-24})}
	\label{prop:h^2}
	If $L_{\polD}$ is hyperbolic, then $H_\Gamma$ exists and $\tD \leq H_\Gamma^2$.
\end{prop}

After this prelude on graphs we can come back to K3 surfaces.
Given a graph $\Gamma$ that is colored with integers in $\{0, \ldots, d\}$ we imitate 
 \cite[Definition~3.6]{DR-2021} (see also  \cite[Theorem~3.9]{DR-2021})
and 
call it {\bf geometric in degree $\tD$}
 if and only if there exists 
a birationally
quasi-polarized  K3 surface  $(X,H)$ of degree $\tD$
such that 
\begin{equation} \label{eq:realization}
\Gamma \cong \Gfamma \, .
\end{equation}
Observe that, for $d=1$,  the graph we call {\sl geometric} 
is {\sl $2$-geometric} in the sense of  \cite[Definition~3.6]{DR-2021}. Still, this simplification 
 leads to no ambiguity in the sequel. 

Finally, 
as an  immediate consequence
of Proposition~\ref{prop:h^2} we obtain the following corollary (cf.\ \cite[Corollary~6.4]{RS-24}) that we will use in the proof of  Theorem~\ref{thm}~(ii).

\begin{cor} 
	\label{cor:finite}
	If
	$\Gamma$ is hyperbolic, then 
	there are at most finitely many integers $h$ such that 
	the graph	$\Gamma$  is geometric in degree $\tD$.
\end{cor}

After those preparations we can prove  Theorem~\ref{thm} (i) and (ii). 

\subsection{Proof of Theorem~\ref{thm}~(i)} 

We  fix an integer $d>0$, and consider the set
	\begin{eqnarray*}
{\mathcal H}&:=&\{h \in \NN_{> 2d^2} \mbox{ such that  there exists a graph }
	 \Gamma \mbox{ with } \mbox{colors in $\{0, \ldots, d\}$}\\
& & \mbox{ and } \#\Gamma \geq 25,  \mbox{ that is geometric in degree $\tD $} 
	 \} \, .
	\end{eqnarray*}
We claim that 
\begin{equation} \label{eq:Hfinite}
 {\mathcal H } \mbox{ is finite}.
\end{equation}
Indeed, by 
Lemmata \ref{lem:elliptic} -- \ref{lem:hyper}, all vertices
of graphs $\Gamma \cong \Gfamma$ that 
appear in the definition of ${\mathcal H }$ are $(-2)$-curves, so the lattices $M(X,H)|_{\SG}$ and $M_{\SG}$ coincide for all subgraphs $G$ of the extended Fano graphs we consider. Observe that Lemmata~\ref{lem:elliptic},~\ref{lem:parabolic} 
combined with the Hodge Index Theorem imply that all Fano graphs in question are hyperbolic.
Moreover,  by Lemma~\ref{lem:extendeddynkin}, for each such  $\Gamma$, we can choose 
a maximal parabolic subgraph  $\Gamma_0\subset\Gamma$
and a vertex $v_0 \in (\Gamma \setminus  \Gamma_0)$. Then, the subgraph $\SG \subset\Gamma$ with vertices in
\begin{equation} \label{eq:enhance}
\Gamma_0 \cup v_0
\end{equation}
 is hyperbolic (it can be neither elliptic nor parabolic). Therefore, Proposition~\ref{prop:h^2} implies that the intrinsic polarization $H_G$ exists and we have 
 \begin{equation} \label{eq:degineq}
 H^2 \leq H_G^2
 \end{equation}
Lemma~\ref{lem:parabolic} yields that there are finitely many parabolic subgraphs of geometric graphs (for a fixed $d$), so there are finitely many possibilities for $\Gamma_0$ in \eqref{eq:enhance}. Moreover,  \eqref{eq-bezout-modified} implies that there are finitely many ways to obtain a hyperbolic $\SG$ from a fixed parabolic $\Gamma_0$ by adding a vertex as in \eqref{eq:enhance}. Thus the set of degrees of intrinsic polarizations of such graphs $\SG$ is finite. By \eqref{eq:degineq}, its maximum gives an upper bound for ${\mathcal H }$. 

Finally, let $h > 2d^2$ exceed the maximum of the set   
${\mathcal H }$ and let $(X,H)$ be  a quasi-polarized K3 surface  of degree $\tD$. Then we have
$$
r_0 + S_d = \# \Gfamma \leq 24
$$
and the proof of Theorem~\ref{thm}~(i) is complete. \hfill $\Box$

\begin{rem}  \label{rem:effective}
	One can check that some slight alterations of the arguments in \cite{RS-24} show that 
	the claim of  Theorem~\ref{thm}~(i) holds for $h > 42 d^2$. We decided not to pursue this issue here to maintain our exposition compact.  
\end{rem}

\subsection{Proof of Theorem~\ref{thm}~(ii)}

Let us fix a positive integer $d$.
By 
 Lemma~\ref{lem:hyper}
 and \eqref{eq-bezout-modified} there are finitely many hyperbolic graphs $\Gamma$ with at most $24$ vertices that are geometric in degree $\tD$ for $h > 2d^2$. Thus, by Corollary~\ref{cor:finite}, there are finitely many integers $h$ such that there exists a
hyperbolic graph $\Gamma$ with at most $24$ vertices that is geometric in degree $\tD$. Consequently
\eqref{eq:Hfinite} implies that the set 
	\begin{eqnarray*}
	{\mathcal H}_{hyp}&:=&\{h \in \NN_{\geq 2d^2} \mbox{ such that  there exists a hyperbolic graph }
	\Gamma \mbox{ with } \\
	& &  \mbox{ colors in $\{0, \ldots, d\}$}  \mbox{ that is geometric in degree $\tD $} 
	\} \, .
\end{eqnarray*}
is finite.
Let $h > 2d^2$ exceed its maximum 
and let $(X,H)$ be  a quasi-polarized K3 surface  of degree $\tD$ with $S_d > 21 - r_0$. By Lemma~\ref{lem:elliptic}, the graph $\Gfamma$ is parabolic. Now
 Theorem~\ref{thm}~(ii) follows directly from Lemma~\ref{lem:parabolic}. \hfill $\Box$

 \section{Proof of Theorem \ref{thm} (iii)}
 \label{s:pf-0}

We shall use elliptic fibrations with section.
Recall that, on a K3 surface $X$, such a fibration is encoded
in an embedding 
$$U\hookrightarrow \NS(X)$$
(thanks to Riemann--Roch),
and that the reducible fibres are encoded in Dynkin diagrams of type ADE 
which are perpendicular to $U$.
The following result will be instrumental,
especially when generalized to positive characteristic.

\begin{theo}
\label{thm:c_0}
Let $p=0$ or $p>3$ be a prime.
For any $c_0\in\ZZ_{\geq 2}$ and  $r\in\{1,\hdots,17\}$
there is a K3 surface $X$ over some field of characteristic $p$
such that
\begin{eqnarray}
\label{eq:NS}
\NS(X)\supseteq U \oplus A_r \oplus \langle -2c_0\rangle.
\end{eqnarray}
Over $\CC$, $X$ can be chosen in such a way that there is only one reducible fibre, namely of type $\IK_{r+1}$,
and all other singular fibres  of the fibration given by 
\eqref{eq:NS} have type $\IK_1$.
\end{theo}
We postpone the proof of Theorem~\ref{thm:c_0} for \S \ref{ss:pf-theo-c_0} and continue by first drawing the desired conclusion
for Theorem \ref{thm}. The necessary generalizations of the second statement to fields of positive characteristic
will be subject of Section \ref{s:pf-p}.

\subsection{Proof of Theorem \ref{thm} (iii)}
\label{ss:pf}

Let $X$ be as in Theorem \ref{thm:c_0}.
Then $X$ admits an elliptic fibration with zero section $O$, as seen above.
By \cite{MWL} and the assumption on the singular fibres, there is a section $P$ of height $2c_0$.
Since the decomposition \eqref{eq:NS} is valid over $\ZZ$,
the section $P$ meets the identity component $\Theta_0$ of the $\IK_{r+1}$ fibre,
i.e.\ the same component as $O$.
In other words, the vector $v$ of square $(-2c_0)$ from \eqref{eq:NS} is given
by the orthogonal projection 
$$\varphi: \NS(X) \to U^\perp$$
applied to $P$:
\begin{eqnarray}
\label{eq:v}
v = \varphi(P) = P - O - c_0 F
\end{eqnarray}
where, following \cite{MWL}, $c_0 = (P.O)+2$ and $F$ denotes the general fiber.

\begin{claim}
\label{claim}
If $d\geq 3$, the vector $H = NF+dO+v$ is quasi-ample and non-hyperelliptic for all $N>\max(2d, c_0+1)$.
\end{claim}

\begin{proof}[Proof of Claim \ref{claim}]
By Criterion \ref{crit}, it suffices to check that $H$ is nef and 
that conditions (i) and (ii) from Criterion \ref{crit} hold true.
Clearly we have
\[
H^2 = 2d(N-d) -2c_0. 
\]
Spelling out the definition of $H$ using \eqref{eq:v}, we have
\[
H = (N-c_0) F + (d-1) O + P.
\]
We thus find that
\begin{itemize}
\item
$H.O = N-2d$,
\item
$H.P = N-c_0 + (d-1)(c_0-2)-2 = N -2c_0 +(c_0-2)d$, 
\item
$H.F = H.\Theta_0 = d \geq 3$;
\item
$H.\Theta_i=0$ for all non-identity components $\Theta_1,\hdots,\Theta_r$ of the $\IK_{r+1}$ fibre;
\item
$H.C \geq (N-c_0)F.C \geq N-c_0$
for any  rational curve $C\neq O,P$
not contained in any fibre (since $C$ forms a multisection);
\item
$H.C \geq (N-c_0)F.C \geq 2(N-c_0)$
for any curve $C$ which is neither smooth rational 
nor contained in any fibre (since $C$ gives a multisection of index strictly greater than one).
\end{itemize}
Taken together, $H$ fulfills Criterion \ref{crit} as soon as 
$N>\max(2d, c_0+1)$.
\end{proof}

We return to the proof of Theorem \ref{thm} (iii).
With $H$ quasi-ample and non-hyperelliptic, it remains to verify
that $r_0 =  r$ -- accounting for the fibre components $\Theta_1,\hdots,\Theta_r$
not intersecting $O$ and $P$ -- 
and that $r_d\geq (24-r)$ -- accounting for $\Theta_0$ and for the $(23-r)$ fibres of type $\IK_1$
(which occur by inspection of the Euler--Poincar\'e characteristic
combined with the generality statement of Theorem \ref{thm:c_0}).

Finally, varying $N>2d$, it suffices to consider a full set of representatives $c_0\in\{2,\hdots,d+1\}$ of $\ZZ/d\ZZ$
to cover all values of $H^2=2h$ starting from $h=(d+1)(d-1)$.
This proves Theorem \ref{thm} (iii).
\qed

\begin{rem}
In \S \ref{ss:iv}, \S \ref{ss:v}, we will adjust the argument to also cover other representatives $c_0'$ 
of $\ZZ/d\ZZ$. 
\end{rem}

\subsection{Warm-up for the proof of Theorem \ref{thm:c_0}}
\label{ss:warm-up}

Consider the case without a section, i.e.\ where 
\begin{eqnarray}
\label{eq:NS'}
\NS(X)\supseteq U \oplus A_r.
\end{eqnarray}
To see that such K3 surfaces form $(18-r)$-dimensional families in any characteristic,
we can argue with the discriminant $\Delta$ of a minimal Weierstrass model.
Locating the special fibre at $t=0$, say,
a fibre of type $\IK_{r+1}$ is given by the codimension $r$ condition
that $v(\Delta)=r+1$, together with the open condition that the fibre is multiplicative, i.e.\ $v(j)=-(r+1)$
for the $j$-invariant (whose numerator thus does not vanish at $t=0$).
In addition, the open condition that $\Delta$ have no other multiple roots
assures that all other singular fibres have type $\IK_1$ as required.

Consider the stratification of the 18-dimensional moduli space of elliptic K3 surfaces with section
into strata $\mathcal P_r$ comprising elliptic K3 surfaces satisfying \eqref{eq:NS'}
(without the extra open conditions on the singular fibres).
It follows that each component of $\mathcal P_r$
\begin{enumerate}
\item
 is either empty or 
 \item
 is non-empty of codimension one
in some component of $\mathcal P_{r-1}$ or, a priori, 
\item
equals some component of $\mathcal P_{r-1}$.
\end{enumerate}
Arguing as in \cite{Artin} or \cite[\S 8.9]{SS-MWL},
we can verify the second alternative for some sequence of components of strata $\mathcal P_i$ up to $\mathcal P_r$
by  verifying the expected dimension of some component $Z$ of $\mathcal P_r$.
Moreover, if the open conditions hold true for some element in $Z$,
then so they do for a general element of the corresponding component of $\mathcal P_i \, (i<r)$.

Presently, start with the extremal rational elliptic surface $S$ with a fibre of type $\IK_9$
(and with $\MW = \ZZ/3\ZZ$, see \cite{Beauville}).
Outside characteristic $3$, this has three further singular fibres, each of type $\IK_1$.
Applying a quadratic base change ramified at the $\IK_9$ fibre,
we obtain a one-dimensional family of K3 surfaces with
$\NS\supset U + A_{17}$, thus exhibiting a component of $\mathcal P_{17}$ of the expected dimension.
Here the embedding is not primitive (due to the torsion section),
but this does not harm the validity of the conclusion that each intermediate stratum $\mathcal P_r \,(r<17)$
has a component of the expected dimension.

\subsection{Strategy of the proof of Theorem \ref{thm:c_0}}
\label{ss:pf-theo-c_0}
\label{ss:strategy}

We now enhance the above discussion by considering the full desired N\'eron--Severi lattice from \eqref{eq:NS}
(possibly up to finite index).
That is, in addition to the singular fibre of type $\IK_{r+1}$ as in \S \ref{ss:warm-up},
we ask for a section $P$ of height $2c_0$.
In practice, 
it is convenient to locate the node of the $\IK_{r+1}$ fibre in a Weierstrass model at $(0,0)$.
Then the coordinates of  $P$ can be expressed by rational functions
\[
P = \left(\dfrac{U}{W^2}, \dfrac{V}{W^3}\right) \;\;\; \text{ where } \;\;\; U,V,W \in k[t]
\]
of degree $2c_0, 3c_0$ resp.\ $(c_0-2)$ such that $(U,W) = (V,W) = 1$ and $t\nmid U$.
Spelling this out against the minimal Weierstrass form, we obtain again a codimension one condition.

Letting the strata $\mathcal M_r$ correspond to the elliptic K3 surfaces satisfying \eqref{eq:NS},
we are in a situation completely analogous to 
\S 
\ref{ss:warm-up},
except that the expected dimension of $\mathcal M_r$ is now $(17-r)$ due to the additional section.
In consequence, Theorem \ref{thm:c_0} can be proved by exhibiting a single K3 surface $X$ 
forming an isolated component of  $\mathcal M_{17}$ (satisfying the given open conditions).

\subsection{Proof of Theorem \ref{thm:c_0} in characteristic zero}
\label{sss}
\label{ss:CC}

Over $\CC$, this is readily achieved:
simply consider the singular K3 surface $X$ with transcendental lattice 
$$
T(X) \cong \langle 2\rangle\oplus \langle 2c_0\rangle.
$$
This exists by \cite{SI},
and it can be regarded as a consequence of the construction in loc.\ cit.\ that
\begin{eqnarray}
\label{eq:fibr'}
\NS(X) \cong U \oplus E_8^2 \oplus A_1 \oplus\langle -2c_0\rangle,
\end{eqnarray}
i.e.\ $X$ admits a jacobian elliptic fibration with two fibres of type $\II^*$, one fibre of type $\IK_2$ and  a section of height $2c_0$.
But then, as exploited in \cite{Nishi}, $X$ admits another jacobian elliptic fibration with a fibre of type $I_{18}$, $3$-torsion in 
the Mordell--Weil group and still a section of height $2c_0$:
\begin{eqnarray}
\label{eq:A_17'}
\NS(X) \cong U \oplus (A_{17})' \oplus\langle -2c_0\rangle;
\end{eqnarray}
here $A_{17}'$ indicates the index 3 overlattice obtained by adding a vector $w\in A_{17}^\vee$ of square $w^2=-4$
(corresponding to the $3$-torsion section).
To verify that $X$ forms the required terminal object,
it remains to check that the configuration of singular fibres is exactly $1 \times \IK_{18} + 6 \times \IK_1$.
Presently the fibre of type $\IK_{18}$ is encoded in \eqref{eq:A_17'}, and together with the $3$-torsion section,
it implies that $X$ arises from the extremal rational elliptic surface $S$ from \S \ref{ss:warm-up}
by quadratic base change. But then the other singular fibres can only be multiplicative, of type $\IK_1$ or $\IK_2$.
Since the latter type does not feature as an orthogonal summand in \eqref{eq:A_17'}, it is excluded.
This completes the proof of Theorem \ref{thm:c_0} over $\CC$.
\qed

\subsection{Proof of Theorem \ref{thm:c_0}  in positive characteristic}

It follows from the Shioda--Inose structure introduced in \cite{SI}
that the singular K3 surface $X$ from \S \ref{sss} is defined over some number field $K$
(cf.\ also \cite{S-fields}).
Hence we can consider the reduction $X_\p$ modulo any prime $\p\subset \mathcal O_K$ dividing $p$.
In fact, by \cite[Cor.\ 0.5]{Matsumoto},
$X$ has potentially good reduction, so this is well-defined upon enlarging $K$,
and we obtain the required embedding
\begin{eqnarray}
\label{eq:spec}
\NS(X) \hookrightarrow \NS(X_\p).
\end{eqnarray}
This completes the proof of Theorem \ref{thm:c_0}
 over fields of positive characteristic.
 \qed

\section{Proof of Theorem \ref{thm} (iv), (v)}
\label{s:pf-p}

In the above set-up,
it  follows from \cite{Shimada-T} that $X_\p$ is supersingular unless
the Legendre symbol satisfies
\begin{eqnarray}
\label{eq:Legendre}
\left(\frac{-c_0}{p}\right) = 1, \;\;\;
\text{ i.e.\   $-c_0$ is a non-zero square mod $p$.}
\end{eqnarray}

\begin{lemm}
\label{lem:deform}
$X_\p$ does not deform together with the sublattice $L=\NS(X)$ of $\NS(X_\p)$ given by \eqref{eq:spec}.
\end{lemm}

\begin{proof}
To verify that $X_\p$ does not deform,
consider the rank 20 sublattice $L\subset \NS(X_\p)$ given by \eqref{eq:spec}.
Either $X_\p$ has finite height, i.e.\ $\rho(X_\p)=20$,
then there are no projective deformations following work of Deligne \cite{Deligne} (cf.\ also \cite{LM}),
or $X_\p$ is supersingular, i.e.\ $\rho(X_\p)=22$.
In the latter case, the moduli space of supersingular 
K3 surfaces is stratified in terms of the Artin invariant 
$\sigma_0\in\{1,\hdots, 10\}$, with each stratum of dimension $\sigma_0-1$ by \cite{Artin}.
But then $\sigma_0$ is defined by the condition that
\[
\NS(X_\p)^\vee/\NS(X_\p) \cong (\ZZ/p\ZZ)^{2\sigma_0}.
\]
Presently we have the rank 20 sublattice $L\cong\NS(X)$ via \eqref{eq:spec};
since 
$$L^\vee/L \cong \NS(X)^\vee/\NS(X)\cong (\ZZ/2\ZZ) \times (\ZZ/2c_0\ZZ),$$
this has $p$-length $0$ or $1$,
so taken together with its rank two orthogonal complement,
we cannot get $p$-length greater than  3 (for a detailed argument, see \cite[Thm.\ 6.1]{KS}). Hence $\sigma_0=1$, and $X_\p$ is isolated in moduli, 
given the rank 20 sublattice $L$ embedded via \eqref{eq:spec}.
\end{proof}

We emphasize that the embedding \eqref{eq:spec} need not be primitive if $K\neq \QQ$
(cf.\ \cite[Ex.\ 3.12]{Maulik-Poonen} and \cite[Thm.\ 4.1.2.1]{Raynaud}).
This is the main subtlety why Lemma \ref{lem:deform} does not  directly imply the next lemma, and thus the  analogue of the second claim of Theorem \ref{thm:c_0} in positive characteristic. 

\begin{lemm}
\label{lem:isol}
For $c_0'\in\ZZ_{>1}$ as detailed in the proof, 
$X_\p\in \mathcal M_{17}$ satisfies the generality assumptions on the singular fibres.
\end{lemm}

\begin{proof}
After Lemma \ref{lem:deform}, it remains and suffices to prove that the elliptic fibration $\pi$ on $X$ encoded in \eqref{eq:NS}
does not degenerate on $X_\p$.
As we have seen in \S \ref{sss}, the only possible degeneration as a K3 surface
consists in attaining a fibre of type $\IK_2$.
In fact, once this is excluded, there surely is a section of height $2c_0$ on $X_\p$,
and Lemma \ref{lem:isol} follows.

To understand the degeneration, we switch to the fibration $\pi'$ from \eqref{eq:fibr'},
or more precisely, to its general form
with fibres of type $\II^*$ twice and $\IK_2$ once.
This one-dimensional family arises from the product $E\times E$ 
of an elliptic curve with itself by way of a Shioda--Inose structure.
Independent of the characteristic $p>3$,
the fibration $\pi'$ degenerates to attain two $\IK_2$ fibres exactly when $j(E)=12^3$,
and to be isotrivial with a $\IV$ fibre exactly when $j(E)=0$. 
The latter case, however, does not cause the fibration $\pi$ to degenerate
as it results in a section of height $3/2$, so we shall concentrate on the first case.

Presently, $X$ arises from the elliptic curve $E_0$ with period $\sqrt{-c_0}$,
which has CM by the order $\ZZ[\sqrt{-c_0}]$ by \cite{SI}.
We are thus led to check whether, at some or all primes $\p$ of $K$ dividing $p$,
$j(E) \equiv 12^3 \mod \p$.
We distinguish the relevant cases and denote the elliptic curve with j-invariant $12^3$ by $E'$.

\subsection{$p\equiv 1\mod 4, (-c_0/p)=0$ or $-1$}

This case is very easy: since $E'$ is ordinary, but $X_\p$ is supersingular by \eqref{eq:Legendre},
we cannot have $j(E) \equiv 12^3 \mod \p$.
Hence the fibres do not degenerate, as desired.
\qed

\subsection{$p\equiv 1\mod 4, (-c_0/p)=1$}
\label{sss4}

In this case, $\rho(X_\p)=20$ by \eqref{eq:Legendre}, so the inclusion \eqref{eq:spec} is of finite index.
Here we can only have $\NS(X_\p) \cong U \oplus E_8^2 \oplus A_1^2$
when $c_0=N^2$ is a perfect square.
But then we replace $c_0$ by $c_0+d, c_0+2d, \hdots$ until the resulting number $c_0'$ is not a perfect square,
and adjust the arguments to accommodate $c_0'$ instead of $c_0$.
To bound $c_0'$, we let
\[
p_0 = \min\{q \text{ prime}; \; q\nmid 2d\}.
\]

\begin{claim}
\label{claim2}
$c_0'\leq c_0+\frac{p_0+1}2 d$.
\end{claim}

\begin{proof}[Proof of Claim \ref{claim2}]
By choice of $p_0$, the elements $c_0, c_0+d,\hdots$ cover all of $\FF_{p_0}$.
Since this field contains exactly $(p_0+1)/2$ squares, 
at least one of the $c_0+jd$ for $j=1,\hdots,(p_0+1)/2$ is no square in $\FF_{p_0}$,
so it's neither a perfect square in $\ZZ$.
\end{proof}

\subsection{$p\equiv 3 \mod 4, (-c_0/p)=1$}
\label{ss:31}

Another easy case as $\rho(X_\p)=20$ by \eqref{eq:Legendre}, 
but $c_0$ cannot be a perfect square, since otherwise $(-c_0/p)=0$ or $-1$.
Hence the finite index inclusion  \eqref{eq:spec} shows that the fibration $\pi$ does not degenerate.
\qed

\subsection{$p\equiv 3 \mod 4, (-c_0/p)=0$ or $-1$}
\label{sss5}

In this case, both $E'$ and $X_\p$ are supersingular,
so there is room for degenerating the fibration $\pi$.

If $p\nmid d$, then we can pursue an approach similar to \S \ref{sss4}
and add multiples of $d$ to $c_0$ until we hit a non-square $c_0' = c_0+jd$ modulo $p$.
Claim \ref{claim2} holds to bound $c_0'$ with $p_0$ replaced by $p$.

If $p\mid d$, and also in general without considering the divisibility of $d$ and also the issue of supersingularity, 
we can appeal to a generalization of the Gross-Zagier formula from \cite[Thm. 3]{Kaneko}
stating that $j(E_0)\not\equiv 12^3 \mod \p$ as soon as $p>4\sqrt{c_0}$.
As a consequence, the fibration $\pi$ cannot degenerate if $p>4\sqrt{d+1}$.
\end{proof}

\subsection{Proof of Theorem \ref{thm} (iv)}
\label{ss:iv}

We adjust the arguments from \S \ref{ss:pf} to cover the auxiliary values $c_0'\leq \frac{p_0+3}2d+1$ from \S \ref{sss4}.
This simply amounts to requiring $N>\frac{p_0+3}2d+2$ and covers all 
$h\geq  d((p_0+1)d+1-p_0)/2$
as stated.
\qed

\subsection{Proof of Theorem \ref{thm} (v)}
\label{ss:v}

If $p\nmid d$, then \S \ref{sss5} shows that the argument and the bound from \S \ref{ss:iv} carry over with $p_0$ replaced by $p$.

Generally, if $p>4\sqrt{d+1}$, the fibration $\pi$ cannot degenerate in characteristic $p$ by \S \ref{sss5},
so the bound $h\geq d^2-1$ carries over from \S \ref{ss:pf}.
\qed

\section{Improved bounds}
\label{s:specific}

The results, especially the bounds for $h$ in Theorem \ref{thm} are general and thus need not be optimal.
Here we indicate how to improve on them in specific cases.

\begin{prop}
\label{prop:small_d}
If $d\leq 20$ and the prime $p>3$ satisfies $p\neq 7$,
then the bounds for $h$ from the complex case Theorem \ref{thm} (iii) are also in effect in characteristic $p$.
\end{prop}

\begin{Addendum}
\label{add}
At $p=7$, the same holds true except for $r_0=17, r_d\geq 7$,
but this can be replaced by $r_0=18, r_d\geq 6$.
\end{Addendum}

\begin{proof}
We argue along the lines of Section \ref{s:pf-p},
but now we determine precisely whether the special situation
$$j(E_0)\equiv 0\mod \p$$
from \S \ref{sss4}, \ref{sss5} persists
for \emph{every} prime $\p$ of $\mathcal O_K$ dividing $p$.
Equivalently, for each $c_0\in\{2,\hdots, 21\}$, and for $n=h(-4c_0)$ the class number of $-4c_0$, 
the minimal polynomial $f\in\ZZ[x]$ of $j(E_0)$
satisfies
\begin{eqnarray}
\label{eq:fmodp}
f \equiv (x-12^3)^n \mod p.
\end{eqnarray}
To improve on this criterion, it suffices at a given $c_0$ and $p$ if 
\eqref{eq:fmodp} can be excluded in one of the following settings:
\begin{itemize}
\item
for $\tilde c_0 = c_0/N^2$
in case $N^2\mid c_0$, but $c_0\neq N^2$
(since in this case we reduce to a height $2\tilde c_0$ section and then take its multiple by $N$), or
\item
 if $\tilde c_0 = c_0/N^2\in \frac 14\ZZ$, but $4\tilde c_0 \equiv -1\mod 4$
 (for then we can argue with the singular K3 surface $\tilde X_0$ with transcendental lattice 
 $$T(\tilde X_0) \cong \begin{pmatrix}2 & 1 \\ 1 & (4\tilde c_0+1)/2
 \end{pmatrix};$$ 
 Here the fibration $\pi$ admits a section of height $\tilde c_0$,
 and its $N$-th multiple will do the job as before.).
 \end{itemize}
We then go through all the minimal polynomials $f$ for discriminant $4c_0\geq -84$
and verify that at each prime $p>3, p\neq 7$ there is a root different from $12^3$
for this polynomial
or for one of the $\tilde c_0$ above.
Hence there is some prime $\p\mid p$ in $\mathcal O_K$ such that the fibration $\pi$ on $X_\p$ is non-degenerate
(i.e.\ the analogue of Theorem \ref{thm:c_0} holds, extended to include  $\tilde c_0$),
and we can conclude as in the complex case in \S \ref{ss:pf}.
\end{proof}

\begin{rem}
\label{rem:p=7}
A close inspection reveals that, for $p=7$, the above approach  only causes problems at squares $c_0$ mod $p$.
Of course, this was to be expected since there is only one supersingular curve in characteristic $7$,
namely the one with $j=12^3$.
\end{rem}

\subsection{Special characteristic}

To cover the case of characteristic $p=7$ from Addendum \ref{add},
we proceed in several steps.

\subsubsection{3-isogeny}

First we apply a 3-isogeny to the elliptic K3 surfaces  $X_0$
corresponding to non-squares $c_0$ mod $p$.
Here we start with the  elliptic fibration $\pi$ with $3$-torsion section
and get a 3-isogenous K3 surface $X_0'$ with a section of height $6c_0$ (via pull-back by the dual isogeny)
such that
there is a finite index sublattice
\[
\NS(X_0') \supset U \oplus A_5 \oplus A_2^6 \oplus \langle -6c_0\rangle.
\]
Since this is automatically non-degenerate (as a consequence of \S \ref{ss:31}), it allows to cover the values $3c_0$
along exactly the same lines as in the previous section.

\subsubsection{5-torsion}

To cover the values $c_0=2,8,11$, we consider the family of complex elliptic K3 surfaces
which generally have fibres of types $\IK_{10}, 2 \times \IK_5, 4 \times \IK_1$
and $\MW_\text{tor} \cong \ZZ/5\ZZ$.
As in \S \ref{ss:warm-up}, this arises by quadratic base change
from a modular elliptic surface which is rational;
this implies that singular fibres are always multiplicative outside characteristic $5$,
and the only degeneration could be two fibers of the same type merging.
Using the discriminant, one verifies that generically
\[
T \cong U \oplus \langle 10\rangle.
\]
This already shows that there are members in the family with any
\[
T = \langle 2c_0\rangle \oplus  \langle 10\rangle \;\;\; (c_0>0).
\]
Obviously, for any $c_0>1$, this translates into a section of height $2c_0$.
At the given values of $c_0$, we have ordinary reduction in characteristic $p=7$ by \eqref{eq:Legendre},
so we get the required K3 surface $X_p$ with a finite index inclusion
\[
\NS(X_p) \supset U \oplus A_9 \oplus A_4^2 \oplus  \langle -2c_0\rangle.
\]

\subsubsection{6-torsion}

The above constructions do not cover easily the values $c_0=4, 16$.
These appear to be surprisingly subtle, so we shall take a little detour, using ideas from \cite{GS} and \cite{S-NS20}.

Consider the family of complex elliptic K3 surfaces $X$ with singular fibres of type $\IK_{12}, 2 \times \IK_3, 2\times\IK_2,2\times\IK_1$
and with $\MW_\text{tor}(X) = \ZZ/6\ZZ$.
This arises from the rational elliptic surface $S$ which is the modular elliptic surface for $\Gamma_1(6)$
by a quadratic base change ramified at $\IK_6$ (say at $\infty$) and at a varying place $s\in\PP^1$.

We shall use that $s$ is a Hauptmodul for the modular curve $X^*(6) = X_0(6)/\langle w_2,w_3\rangle$
where we quotient by both Fricke involutions.
To see this, consider the quadratic twist $S'$ of $S$ at $\infty$ and $s$.
By construction, this is the quotient of $X$ by a Nikulin involution (also parametrised by $s$).
It is easy to show that the family of $S'$ agrees with the family $\mathcal X_3$ of K3 surfaces from \cite[\S6]{GS},
for instance by identifying a divisor of Kodaira type $\IK_6^*$ together
with two sections and  perpendicular configurations of types $A_2, A_1, D_4$ in 
\cite[Fig.\ 5]{GS}.
Explicitly, this leads to the relation $s=-2/r$ for the parameter $r$ from \cite[\S6]{GS},
which itself relates to the standard Hauptmodul $a$ of $X^*(6)$ by $a = -2(r + 2)/(4r - 1).$

Recall that we are mainly interested in exhibiting a section $P$ of height $2c_0=8$ 
(since then the case $c_0=16$ follows readily by considering $2P$).
Using the expression for the j-invariants of the  elliptic curves underlying $X^*(6)$ from \cite[(21)]{GS},
we can solve for them to have CM by the ring of integers 
in $K=\QQ(\sqrt{-6})$.
This leads to $r=5\pm 2\sqrt{3}$, giving a (pair of) singular K3 surface $X$ with CM in $K$,
whose precise transcendental lattice escapes us at first,
but we have to be somewhat inventive to get what we need in characteristic $7$
(which a posteriori also allows to compute $T(X)$, see Corollary \ref{cor:T(X)}).

To this extent, consider the reduction modulo $p=7$,
an ordinary K3 surface $X_p$ over $\FF_q$ where $q=p^2$.
One finds that all fibres are non-degenerate (as opposed to the two $\IK_1$ fibres merging,
which would result in discriminant $-24$, cf.\ the entry \cite[Table 1]{GS} (which corresponds to a cusp of $X^*(6)$)).
Moreover, we are free to arrange for the fibres of type $\IK_{12}$ and $\IK_3$ to be located at $\infty$ resp.\ at $\pm 1$.
It turns out that that this makes all singular fibres defined over $\FF_q$; what is more, all fibres are automatically split over $\FF_q$
because they originate from the surface $S$ with singular fibres defined over $\QQ$.
In consequence, 
$$
\rho(X_p/\FF_q) =  19 \;\;\; \text{ or } \;\;\; 20,
$$
 depending on the field of definition of the remaining generator of $\NS(X_p)$,
section $P$ of height a square multiple of $2$ by inspection of the CM field $K$.
The Lefschetz fixed point formula for the action of Frob$_q$ thus returns
\begin{eqnarray}
\label{eq:lef}
\#X_p(\FF_q) = 1 + 19 q \pm q + a_q + q^2.
\end{eqnarray}
Here $a_q$ denotes the trace of Frob$_q$ for the Galois representation induced by the transcendental lattice of $X$ over $\CC$.

We now reduce \eqref{eq:lef} modulo $6$.
Indeed, the left-hand side is congruent to zero modulo $6$
since every smooth fibre $F$ satisfies $6\mid\#F(\FF_q)$ by virtue of the $6$-torsion sections,
and the singular fibres contain $24q$ points in total since they are split over $\FF_q$.
Thus we get
\[
0 \equiv 3 \pm 1 \pm 2 \mod 6,
\]
which translates into the following two cases:
\begin{enumerate}
\item
$P$ is defined over $\FF_q$ and $a_q=2$,
\item
$P$ is defined over $\FF_{q^2}$, but not over $\FF_q$, and $a_q=-2$.
\end{enumerate}
We now employ the Artin-Tate conjecture as in \cite{S-NS20}.
In case (2), this gives the relation
\begin{eqnarray*}
200/q  & = &  [(1-a_qT+q^2T^2)(1+qT)]_{T=1/q} \\
& = & \#\mbox{Br}(X/\FF_q) \cdot |\det\NS(X/\FF_q)|/q \;\; = \;\;  N^2 12/q,
\end{eqnarray*}
where we use that $\#$Br$(X/\FF_q)$ is a square $N^2$ and $\NS(X/\FF_q)$ equals the trivial lattice
comprising fibre components and torsion sections.
But then the above equality is literally impossible.
Therefore case (1) persists, and the same approach gives
\[
96/q = [(1-a_qT+q^2T^2)]_{T=1/q} = \#\mbox{Br}(X/\FF_q) \cdot |\det\NS(X/\FF_q)|/q.
\]
By \cite{MWL}, $X_p$ has to admit a section $P$ of height $2$ or $8$ 
which generates $\NS(X_p)$ (of discriminant $-24$ or $-96$) together with the  trivial lattice.
(This uses the fact that  the singular fibres are non-degenerate on $X_p$, as we checked.)
The cases of $c_0=4, 16$ are thus settled as desired by the following:

\begin{claim}
\label{claim:2}
There is a section of height $8$ meeting each fibre in the identity component.
\end{claim}

\begin{proof}[Proof of Claim \ref{claim:2}]
Assuming there is a section $P$ of height $2$,
the possible contraction terms from \cite{MWL} call for the configuration of 
$\IK_{12}$ being met by $P$ at the component opposite to the identity component,
and both $\IK_2$ at the non-identity component (plus $(P.O)=1$).
But then, the 2-torsion section $Q$ meets exactly the same fibre components,
so the Mordell--Weil pairing gives the contradiction
$$
0=\langle P,Q\rangle = 2 + (P.O) - (P.Q)  - 3 - 1/2 - 1/2 \leq -1.
$$
Hence there is a section $P$ of height $8$ by what we have seen before.
Again, $P$ can only meet non-identity components in the same configuration as above,
but with $(P.O)=4$.
Anyway, $P+Q$ then meets all fibres at the identity components while again having height $8$.
\end{proof}

\begin{cor}
\label{cor:T(X)}
Let $\sigma$ be the conjugation in Gal$(\QQ(\sqrt{-3})/\QQ)$.
Then 
$X\cong X^\sigma$ as complex K3 surfaces,
with transcendental lattice
$T(X) \cong \mbox{diag}(8,12)$.
\end{cor}

\begin{proof}
By \cite[Thm.\ 4.1.2.1]{Raynaud},
the specialization map at $p=7$,
\[
\NS(X) \hookrightarrow \NS(X_p),
\]
is a primitive embedding. Since the ranks are the same by the choice of $X$,
we get an isometry. Thus we compute the discriminant form of $\NS(X)$
to be the opposite of the claimed transcendental lattice.
This is uniquely determined by its discriminant form,
since its genus consists of a single class by the theory of binary quadratic forms.
\end{proof}

\subsubsection{Supersingular reduction lattice}
\label{sss:p=7}

Consider the elliptic modular surface $Y$ for $\Gamma_1(8)$.
This is a K3 surface (defined over $\ZZ[1/2]$) with singular fibres of types
$2 \times \IK_8, \IK_4, \IK_2, 2\times \IK_1$
and $\MW(Y)\supset\ZZ/8\ZZ$.
We shall crucially use the fact that the torsion subgroup causes 
the fibre configuration to be
never degenerate in odd characteristic.
Over $\CC$, the transcendental lattice is
$T(Y) = \mbox{diag}(2,4)$;
at a supersingular prime $p\equiv 5,7\mod 8$,
this is reflected in the narrow Mordell--Weil lattice 
(i.e.\ the orthogonal complement of the trivial lattice $U\oplus A_1\oplus A_3 \oplus A_7^2$ inside $\NS(Y_p)$)
being 
$$
\MWL^0(Y_p) = T[-p],
$$
the transcendental lattice scaled by $-p$.
In particular, this confirms that there are sections, meeting only the identity components of the singular fibres,
of heights $2p, 4p, 6p$ and $8p$.
This gives surfaces with $r_0=18$ and $r_d\geq 6$
at the residue classes $c_0\equiv p, 2p, 3p, 4p \mod d$,
yielding the same bounds for $H^2$ as in Theorem \ref{thm} (iii) at $p=7$, but with $r_0=18$ instead of $r_0=17$.

To prove Addendum \ref{add} for these values of $c_0$ (in fact, in a unified way for all $p\equiv 5,7\mod 8$ with $p>5$)
it remains to verify that from the terminal object $Y_p$,
we can get all intermediate pairs $(r_0, r_d) = (r, 24-r)$ with $r\in\{1,\hdots,16\}$.
To see this, we start as in \S \ref{ss:warm-up} with an auxiliary series of primitive embeddings of root lattices $R_i$ of respective rank $i$,
\begin{eqnarray*}
A_1 = R_1 & \hookrightarrow  &
R_2 \; \hookrightarrow \; \hdots \;  \hookrightarrow  \; R_{16} = A_3^5\oplus A_1\\
&  \hookrightarrow  & R_{17}  = 
 A_1 
\oplus A_3^3 \oplus A_7
\; \hookrightarrow \;  A_1 
\oplus A_3 \oplus A_7^2 = R_{18},
\end{eqnarray*}
where the last lattice corresponds to the fibre components of $Y_p$.
We can stratify the moduli space of jacobian elliptic 
K3 surfaces into strata $\mathcal P_i'$
such that $U\oplus R_i\hookrightarrow\NS$.
As in \S \ref{ss:warm-up}, each inclusion is componentwise of codimension one -- because $Y_p$ does not deform with $U\oplus R_{18}$
by the very same arguments as in Lemma \ref{lem:deform}.
That is, $\dim \mathcal P_i' = 18-i$.
Now we enhance with what is meant to correspond to a section of height $2c_0$ and stratify by $\mathcal M_i'$ such that
\begin{eqnarray}
\label{eq:R_i+}
U \oplus R_i \oplus\langle -2c_0\rangle \hookrightarrow\NS.
\end{eqnarray}
Since $Y_p\in \mathcal M_i'\neq\emptyset$, 
only the latter 2   alternatives from \S \ref{ss:warm-up} are in effect for the inclusions $\mathcal M_i'\subseteq \mathcal P_i'$.
In particular, each component of $\mathcal M_{16}'$ has dimension $1$ or $2$.

\begin{claim}
\label{claim-P}
For each component $Z\subset\mathcal M_{16}'$, we have $\dim Z = 1$.
\end{claim}

\begin{proof}[Proof of Claim \ref{claim-P}]
Assume that $\dim Z = 2$ and let $W\in Z$.
Since $\rho(W)\geq 19$, the K3 surface $W$ cannot have finite height,
else it would only deform with the rank $19$ sublattice from \eqref{eq:R_i+} in a one-dimensional family.
Hence $\rho(W)=22$, but then supersingular K3 surfaces can only deform in a two-dimensional family
if generally the Artin invariant $\sigma_0\geq 3$.
This, however, is ruled out by a $p$-length argument as in the proof of Lemma \ref{lem:deform}.
\end{proof}

Note that Claim \ref{claim-P} implies that every component $Z'\subset\mathcal M_i'$ has the expected dimension $18-i$ for any i.

We conclude the proof of Addendum \ref{add} for the given values of $c_0$
by verifying that there is K3 surface $Y'\in\mathcal P_{16}'$ with non-degenerate fibres in characteristic $p$.
First, all singular fibres are multiplicative due to the torsion section.
Secondly, we have to show that the reducible fibres are exactly given by the orthogonal summands of $R_{16}$,
i.e.\ 5 times $\IK_4$ and once $\IK_2$.

Assume to the contrary that this does not hold true for a generic member $Y_\eta$ of $\mathcal P_{16}'$.
Since $R_{16}\oplus \langle-2c_0\rangle$ does not contain a root overlattice of finite index (essentially because $p\mid c_0$
and $p>5$),
we deduce that there is a root lattice $R_{17}'$ such that
\begin{eqnarray}
\label{eq:Y_eta}
U \oplus R_{16}\hookrightarrow
U \oplus R_{17}' \hookrightarrow \NS(Y_\eta)
\end{eqnarray}
where both inclusion have corank one or more.
In particular, $\rho(Y_\eta)\geq 20$, 
so that in fact $\rho(Y_\eta)=22$
due to the one-dimensional family which $Y_\eta$ lives in.
But then, by going through all possible rank 17 root lattices $R_{17}'$ containing $R_{16}$, we find that $p\nmid\det(R_{17}')$
as soon as $p>5$. 
Hence the proof of Lemma \ref{lem:deform} strikes again to show that $\sigma_0=1$,
contradicting the 1-dimensional space of deformations given by $\mathcal P'_{16}$.
Therefore, the generic member $Y_\eta$ of $\mathcal P_{16}'$ has non-degenerate singular fibres as claimed,
and Addendum \ref{add} follows for $c_0 \equiv p, 2p, 3p, 4p \mod d$ at $p=7$
(and at all $p\equiv 5,7\mod 8, \, p>5$ alike).
\qed

\subsubsection{Conclusion}

With all $c_0\leq 21$ covered in characteristic $p=7$, we can prove Addendum \ref{add}
exactly as in \S \ref{ss:CC} -- with the same  bounds, but with $r_0=18, r_d\geq 6$ replacing $r_0=17, r_d\geq 7$
in the cases with $p\mid c_0$ from \S \ref{sss:p=7}.

\begin{rem}
The same approach also works for other characteristics
to improve on the bounds from Theorem \ref{thm} (iv) and (v).
\end{rem}

\section{Unconditional results}

We conclude this paper by indicating how to derive unconditional results
with better bounds for $h$ if we content ourselves with fewer degree zero curves.
As a side effect, the result also covers characteristic $3$.

\begin{theo}
\label{thm:r<=14}

Let $d\geq 3, \, p \equiv 1 \mod 4$ and $h\geq 2d^2+d+1$.
For any   
$r\in\{1,\hdots, 14\}$, 
there are K3 surfaces of degree $2h$ over some field of characteristic $p$ such that 
$$
r_d\geq 24-r \;\;\;  \mbox{ and } \;\;\;  r_0 =r.$$
For $p\equiv 3\mod 4$, the same holds true for any $r\in\{1,\hdots, 13\}$.
\end{theo}

\begin{proof}
We combine the approaches from \cite[Thm.\ 10.1]{RS-24} and from Theorem \ref{thm:c_0}
starting from a single terminal object,
namely the  elliptic K3 surface $Y$
with singular fibres of type $\IK_8$ twice and $\IK_2$ four times.
This has 
$$\MW(Y)_{\text{tor}}\cong \ZZ/4\ZZ \times\ZZ/2\ZZ$$
and is given explicitly by a cyclic degree 4 base change from the Legendre curve:
\[
Y: \;\;\; y^2 = x (x-1) (x-t^4).
\]
One directly verifies that the fibration never degenerates modulo $p>2$.
As in \eqref{eq:Legendre}, the reduction $Y_p=Y\otimes\bar\FF_p$ satisfies
\[
\rho(Y_p) = 
\begin{cases}
20
& \text{ if } p\equiv 1\mod 4,\\
22
& \text{ if } p\equiv 3\mod 4.
\end{cases}
\]
We will need the following auxiliary result resembling \cite[Lem.\ 10.1]{RS-24}:

\begin{lemm}
\label{lem:aux}
Fix $p,r,d$ as in Theorem \ref{thm:r<=14}.
For any $c_0\in\{0,\cdots,d-1\}$, there is a root lattice $R$, comprising Dynkin diagrams solely of type $A$, 
and a family of K3 surfaces in characteristic $p$, generically ordinary of rank $r+2$,
with 
\begin{eqnarray}
\label{eq:embedding}
\begin{pmatrix}
0 & d\\
d & -2c_0
\end{pmatrix} \oplus R
\hookrightarrow \NS.
\end{eqnarray}
\end{lemm}

\begin{proof}[Proof of Lemma \ref{lem:aux}]


Write $r=r_1+r_2+r_3$ with $0\leq r_i\leq 7$ such that $p\nmid (r_i+1)$ for each $i$.
In case all $r_i>0$, we also have to assume that $r_2+r_3<7$
(so that $A_{r_2}\oplus A_{r_3} \hookrightarrow A_7$).
Let $c_0=a_1^2+\hdots+a_4^2$ for $a_i\in\NN_0$.
We would like to deform $Y_p$ preserving the genus one fibration together
with (multiplicative) fibres of type $A_{r_1},\hdots,A_{r_3}$ (deforming the two $A_7$ fibres and giving $R$)
and with the divisor $D = d(F+O)+a_1\Theta_1+\hdots a_4\Theta_4$,
where the $\Theta_i$ denote the non-identity components of the four $\IK_2$ fibres of $Y$.

If $p\equiv 1\mod 4$, this works readily by \cite{Deligne} because $Y_p$ is ordinary.

If $p\equiv 3\mod 4$, we  first deform $Y_p$ a little
by imposing that there are two fibres of type $\IK_4$ instead of one of the $\IK_8$ fibres.
Using the arguments from \S \ref{ss:warm-up} and \S \ref{ss:strategy},
we infer that this gives a one-dimensional family of K3 surfaces,
generally ordinary, such that 
\[
U \oplus A_7\oplus A_3^2 \oplus A_1^4 \hookrightarrow \NS.
\]
Starting from a general ordinary member, we can continue as before (with $r_2, r_3\leq 3$) 
at the expense of excluding the maximal value $r=14$.
\end{proof}

\begin{rem}
If $c_0$ is a sum of less than 4 squares, then we can preserve some of the $\IK_2$ fibres
in the above deformation argument
and increase the range of possible $r$ accordingly.
\end{rem}

To complete the proof of Theorem \ref{thm:r<=14}, 
we consider a general member $Y_0$ of the family of K3 surfaces from Lemma \ref{lem:aux};
while the embedding \eqref{eq:embedding} need not be primitive,
we can assure this for the isotropic vector $F$ which deforms from $Y_p$ (where it was obviously primitive).
It follows from Riemann--Roch, that $F$ is either effective or anti-effective,
so let us assume the former after switching the sign, if necessary.
Then $|F|$ may still involve some base locus, consisting of $(-2)$-curves,
but this is eliminated by a composition $\sigma$ of
Picard-Lefschetz reflections.
In consequence, $F_0 = \sigma(F)$ is the fibre of an elliptic fibration on $Y_0$ given by $|F_0|$.

It follows that $\sigma(R)$ is supported on effective or anti-effective $(-2)$-divisors perpendicular to $F_0$,
i.e.\ $\sigma(R)$ is supported on fibre components.
Thus, by the general choice of $Y_0$ deforming $Y_p$, we may assume that 
all reducible fibres continue to be multiplicative, 
of Kodaira types $\tilde R_v$ given by the orthogonal summands $R_v$ of $R$ (or of $\sigma(R)$),
and all other singular fibres have type $\IK_1$.

It remains to investigate the divisor $D_0 = \sigma(D+F)$ of square $D_0^2=2(d-c_0)>0$.
Again by Riemann--Roch,  $D_0$ is either effective or anti-effective,
but since $D_0.F_0=d>0$ the first alternative persists.
The support of $D_0$ thus comprises multisections and fibre components.
Here we may assume that $\Theta.D_0\geq 0$ for any  fibre component $\Theta$.
Indeed, if $\Theta.D_0<0$, then
the reflection $\sigma_{\Theta}$ in $\Theta$ results in $0\leq \sigma_\Theta(D_0)<D_0$,
and so we can proceed successively for all such fibre components,
modifying $\sigma(R)$ as well, but not affecting $F_0$.

\begin{claim}
\label{claim8}
For any $N\geq2d$, the divisor $H_0 = NF + D_0$ is quasi-ample and non-hyperelliptic.
\end{claim}

\begin{proof}[Proof of Claim \ref{claim8}]
We use Criterion \ref{crit} to first check that $H$ is  nef and then that it is also quasi-ample and non-hyperelliptic.
To this extent, we estimate $H_0.C$ for any irreducible curve $C\subset Y_0$:
\begin{itemize}
\item
$H_0.C=d\geq 3$ if $C\sim F_0$;
\item
$H_0.C\geq 0$ for any smooth rational fibre component by choice of $\sigma$;
\item
$H_0.C\geq N(F_0.C) \geq 2N$
for any multisection $C$ which is not smooth rational
(so the multisection  index satisfies $F_0.C>1$);
\item
if $C$ is a smooth rational multisection, then $C$ has multiplicity at most $d$ in $D_0$
(since $D_0.F_0=d$ and $F_0$ is nef),
and $H_0.C\geq (NF_0+dC).C = N-2d$.
\end{itemize}
It remains to compute $H_0^2 = 2Nd+2d-2c_0\geq 4d^2+2\geq 38$
to deduce the claim of Criterion \ref{crit}.
\end{proof}

The proof of Theorem \ref{thm:r<=14} concludes by letting $N>2d$,
so that, by the proof of Claim \ref{claim8},
the only degree 0 curves are fibre components.
More precisely, since $H_0.\Theta\geq 0$ for any fibre component,
but $H_0\perp\sigma(R)$, we infer that $H_0$ meets exactly one component of each fibre with multiplicity $d$
(the component not contained in the support of $\sigma(R)$ -- for rank reasons, there cannot be more than one component
outside the support).
That is, there are exactly $r= $ rank $R$ components of degree zero (so $r_0=r$)
and at least $24-r$ rational components of degree $d$ 
-- the components off the support of $\sigma(R)$ and the fibres of type $\IK_1$ (so $r_d\geq 24-r$ as stated).

Varying $N>2d$, this covers $H_0^2=2h$ for all $h\geq 2d^2+d+1$.
\end{proof}

\subsection*{Acknowledgement}

This work was started at the conference 'Aspects r\'eels de la g\'eom\'etrie' at CIRM
on the occasion of Alex Degtyarev's 60th birthday.
We thank the organizers for the kind invitation,
CIRM for the great hospitality, and the birthday boy for many inspiring discussions throughout the years.

\end{document}